\documentclass[a4paper,oneside,11pt]{amsart}

\reversemarginpar

\usepackage[colorlinks,hyperindex,linkcolor=blue,urlcolor=black,pdftitle={Examples of cyclic polynomially bounded operators that are not similar to contractions, II},pdfauthor={Maria F.Gamal'}]
{hyperref}

\usepackage{amsmath}
\usepackage{amsfonts}
\usepackage{amssymb}
\usepackage{amsthm}

\usepackage[english]{babel}
\usepackage{enumerate}

\usepackage{graphicx}
\usepackage{color}

\usepackage[abbrev]{amsrefs}

\numberwithin{equation}{section}

\theoremstyle{plain}
\newtheorem{theorem}{Theorem}[section]
\newtheorem{lemma}[theorem]{Lemma}

\newtheorem{proposition}[theorem]{Proposition}

\theoremstyle{definition}
\newtheorem{remark}[theorem]{Remark}

\begin{document}

\title[Examples of cyclic polynomially bounded operators, II]{Examples of cyclic polynomially bounded operators that are not similar to contractions, II}

\author{Maria F. Gamal'}
\address{
 St. Petersburg Branch\\ V. A. Steklov Institute 
of Mathematics\\
 Russian Academy of Sciences\\ Fontanka 27, St. Petersburg\\ 
191023, Russia  
}
\email{gamal@pdmi.ras.ru}


\subjclass[2010]{ Primary 47A60; Secondary 47A65,  47A16, 47A20.}

\keywords{Polynomially bounded operator, similarity, contraction, 
unilateral shift, isometry, $C_0$-contraction, $C_0$-operator}


\begin{abstract}
The question if polynomially bounded operator is similar to a contraction was posed by Halmos and was answered in the negative by Pisier. His counterexample is an operator of infinite multiplicity, while all its restrictions 
on invariant subspaces of finite multiplicity are similar to contractions. 
In \cite{gam16}, cyclic  polynomially bounded operators which are not similar to contractions was 
constructed. The construction was based on a perturbation of the sequence of finite dimensional operators which is uniformly polynomially bounded,  but is not uniformly completely polynomially bounded, constructed by Pisier. 
In this paper, a  cyclic  polynomially bounded operator $T_0$ such that $T_0$ is not similar to a contraction and 
$\omega_a(T_0)=\mathbb O$, is constructed. Here $\omega_a(z)=\exp(a\frac{z+1}{z-1})$, $z\in\mathbb D$, $a>0$, and 
$\mathbb D$ is the open unit disk. To obtain such $T_0$, a slight modification of the construction from  \cite{gam16} is needed. 

\end{abstract} 

\maketitle

\section{Introduction}

Let $T$ be an operator on a Hilbert space.  
$T$ is called {\it polynomially bounded}, if there exists a constant $C$
such that $$\|p(T)\|\leq C\|p\|_\infty \text{ \ for every polynomial } p,$$ 
where $\|\cdot\|_\infty$ denotes the uniform norm on the open unit disk $\mathbb D$.

An operator $T$ is called a {\it contraction} if $\|T\|\leq 1$. Every contraction is polynomially bounded (von Neumann inequality, see, for example, {\cite[Proposition I.8.3]{sfbk}}).

The question if a  polynomially bounded 
operator is similar  
to a contraction was posed by Halmos \cite{halmos} and was answered in the negative 
by Pisier \cite{pisier}.  In particular, the sequence of finite dimensional operators which is uniformly polynomially bounded, but is not uniformly completely polynomially bounded was constructed in \cite{pisier} (see also \cite{dp}; see \cite{paulsen},  \cite{pisier}, or  \cite{dp} for the definition of  completely polynomially boundedness). 

In {\cite[Theorem 6.3]{gam16}}, a perturbation of the sequence described above  was constructed.  The obtained perturbed sequence was used to construct examples of  cyclic  polynomially bounded operators which are not similar to contractions 
({\cite[Theorem 7.1 and Sec. 2]{gam16}}). In this paper, a slight modification  of {\cite[Theorem 7.1]{gam16}} 
is given which allows to construct a cyclic  absolutely continuous polynomially bounded operator $T_0$ such that $T_0$ is not similar to a contraction and 
$\omega_a(T_0)=\mathbb O$, where $a>0$ and $\omega_a(z)=\exp(a\frac{z+1}{z-1})$ ($z\in\mathbb D$) 
(Theorem~\ref{example}). 

We need the following notations and definitions. 

Let $\mathcal H$ be a (complex, separable) Hilbert space, and let $\mathcal M$ be its
(linear, closed)  subspace.
By $I_{\mathcal H}$ and $P_{\mathcal M}$ the identical operator on $\mathcal H$ and 
the orthogonal projection from $\mathcal H$ onto $\mathcal M$ are denoted, respectively.  

For an operator  $T\colon \mathcal H\to \mathcal H$, a subspace $\mathcal M$  of $\mathcal H$ is called \emph{invariant subspace} of $T$, if $T\mathcal M\subset\mathcal M$. The complete lattice of all invariant  subspaces of $T$ is denoted by  $\operatorname{Lat}T$.

The {\it multiplicity} $\mu_T$ of an operator  $T\colon \mathcal H\to \mathcal H$ 
 is the minimum dimension of its reproducing subspaces: 
$$  \mu_T=\min\{\dim E: E\subset \mathcal H, \ \ 
\bigvee_{n=0}^\infty T^n E=\mathcal H \} .$$
An operator $T$ is called {\it cyclic}, if $\mu_T=1$. 

It is well known and easy to see that if $\mathcal M\in\operatorname{Lat}T$, then 
\begin{equation}\label{muorth} \mu_{P_{\mathcal M^\perp}T|_{\mathcal M^\perp}}\leq\mu_T\leq
  \mu_{T|_{\mathcal M}} + \mu_{P_{\mathcal M^\perp}T|_{\mathcal M^\perp}}  \end{equation}
(see, for example, {\cite[II.D.2.3.1]{nik}}).

Let $T$ and $R$ be operators on Hilbert spaces $\mathcal H$ and $\mathcal K$, 
respectively, and let
$X:\mathcal H\to\mathcal K$ be a (linear,  bounded) operator such that 
$X$ {\it intertwines} $T$ and $R$,
that is, $XT=RX$. If $X$ is unitary, then $T$ and $R$ 
are called {\it unitarily equivalent}, in notation:
$T\cong R$. If $X$ is invertible, that is, $X$ has the {\it bounded} 
inverse $X^{-1}$, 
then $T$ and $R$ 
are called {\it similar}, in notation: $T\approx R$.
If $X$ is a {\it quasiaffinity}, that is, $\ker X=\{0\}$ and 
$\operatorname{clos}X\mathcal H=\mathcal K$, then
$T$ is called a {\it quasiaffine transform} of $R$, 
in notation: $T\prec R$. If $T\prec R$ and 
$R\prec T$,
 then $T$ and $R$ are called {\it quasisimilar}, 
in notation: $T\sim R$. 
If $\operatorname{clos}X\mathcal H=\mathcal K$, we write $T \buildrel d \over \prec R$. 
It  follows immediately from the definition of the relation $\buildrel d \over \prec$ that if $T \buildrel d \over \prec R$,
then $\mu_R\leq \mu_T$.

If $T$, $R$, and $X$ are operators such that $XT=RX$, and $\mathcal M\in\operatorname{Lat}T$, then 
$\operatorname{clos}X\mathcal M\in\operatorname{Lat}R$.
The mapping  
\begin{equation}\label{lattice}\mathcal J_X\colon\operatorname{Lat}T\to \operatorname{Lat}R, \   
\ \mathcal J_X\mathcal M=\operatorname{clos}X\mathcal M, \ \ \ \mathcal M\in\operatorname{Lat}T,\end{equation} 
is a lattice isomorphism if and only if it is a bijection, see  {\cite[VII.1.19]{berbook}}.

Symbols $\mathbb D$  and $\mathbb T$ denote the open unit disc  
and the unit circle, respectively. The normalized Lebesgue measure on $\mathbb T$ is denoted by $m$.  The symbol 
 $H^\infty$ denotes the Banach algebra of all bounded 
analytic functions in $\mathbb D$. The uniform norm on $\mathbb D$
is denoted by the symbol $\|\cdot\|_\infty$. The symbol $H^2$ denotes 
the Hardy space on $\mathbb D$, $L^2=L^2(\mathbb T,m)$, $H^2_-=L^2\ominus H^2$. The symbol $P_-$ denotes the orthogonal projection from $L^2$ onto $H^2_-$. 
 The simple unilateral and bilateral shifts $S$ and $U_{\mathbb T}$ are the operators
 of multiplication by the independent variable  on $H^2$ and on $L^2$, respectively. It is well known that $\mu_S=1$ and $\mu_{U_{\mathbb T}}=1$.
For an inner function  $\theta\in H^\infty$ the subspace $\theta H^2$
is invariant for $S$, put 
\begin{equation}\label{1.7}\mathcal K_\theta=H^2\ominus\theta H^2 \ \ \text{ and } \ \ 
T_\theta = P_{\mathcal K_\theta}S|_{\mathcal K_\theta}.  \end{equation}
By \eqref{muorth}, $\mu_{T_\theta}=1$.

If $T$ is a polynomially bounded operator, then $T=T_a\dotplus T_s$,
where $T_a$ is an  \emph{absolutely continuous (a.c.)}  polynomially bounded operator, 
that is, $H^\infty$-functional calculus is well defined for $T_a$, and 
$T_s$ is similar to a singular unitary operator, see \cite{mlak} or \cite{ker16}. (Although many results on polynomially 
bounded operators that will be used in the present paper were originally  proved by Mlak, we will refer to \cite{ker16} for the convenience of references.) 
In the paper, absolutely continuous polynomially bounded operators 
are considered.  Clearly,  $S$ and $U_{\mathbb T}$ are contraction, and, consequently, polynomially bounded. It is well known that $S$ and $U_{\mathbb T}$ are a.c..
An a.c. polynomially bounded operator $T$ is called a
{\it $C_0$-operator}, if there exists $\varphi\in H^\infty$ such that 
$\varphi(T)=\mathbb O$, see \cite{bercpr}; if a $C_0$-operator is a contraction, it
is called a {\it $C_0$-contraction}, see \cite{sfbk} and \cite{berbook}. For an inner function $\theta$, 
$T_\theta$ is a $C_0$-contraction, because $\theta(T_\theta)=\mathbb O$.  

The paper is organized as follows. In Sec. 2  polynomially bounded operators that are quasiaffine transforms of the unilateral shift $S$ are considered. In many results is assumed that the operator $T$ under consideration has an invariant subspace $\mathcal M$  such that $T|_{\mathcal M}\approx S$. If $T\prec S$ and $T$ is a contraction, then $T$ 
always has such  invariant subspace  (a particular case of \cite{ker07}, see also {\cite[Theorem IX.3.5]{sfbk}}). The author does not know if this result can be generalized on polynomially bounded operators. In Sec. 3
Theorem \ref{theorem1} is applied to construct the  polynomially bounded operator $T_0$ such that $T_0$ is not similar to a contraction and 
$\omega_a(T_0)=\mathbb O$, where $a>0$ and $\omega_a(z)=\exp(a\frac{z+1}{z-1})$ ($z\in\mathbb D$). (All subsequent results from Sec. 2 are not used in Sec. 3.) 


\section{On operators that are quasiaffine transforms of a unilateral shift}

An  operator $T$ is called  \emph{power bounded}, if $\sup_{n\geq 0}\|T^n\| < \infty$.
For a power bounded operator $T$ the  \emph{isometric asymptote} $(Y,V)$  is defined in \cite{ker89}.  Here $V$ is an isometry, and $Y$ is an intertwining mapping: $YT=VY$. 
When it will be convenient, we will called by the isometric asymptote of $T$ the isometry $V$ itself, while $Y$ will be called 
the  \emph{canonical intertwining mapping} (sf. \cite{ker89} and \cite{ker16}).

Let $\mathcal M\in\operatorname{Lat}T$.  It follows immediately from  the construction of the isometric asymptote and {\cite[Theorem 1]{ker89}} that   $T|_{\mathcal M}$ is similar to an isometry if and only if $Y|_{\mathcal M}$ is left invertible, i.e., there exists $c>0$ such that $\|Yx\|\geq c\|x\|$ for every $x\in\mathcal M$. The  \emph{unitary asymptote} for a power bounded operator $T$ is the minimal unitary extension of  its isometric asymptote, see \cite{ker89}.   

Clearly, a polynomially bounded operator is power bounded. The unitary asymptotes for polynomially bounded operators are considered in \cite{ker16}.
 
In the following proposition, the idea from \cite{fer} is used. Since this proposition can be easily checked directly, its proof is omitted.

\begin{proposition}\label{propa1} Suppose that  $T\colon\mathcal H\to\mathcal H$, $V\colon\mathcal K\to\mathcal K$,  and $Y\colon\mathcal H\to\mathcal K$  are operators, and $YT=VY$. Furthermore, suppose that  $\mathcal M\in\operatorname{Lat}T$ is such that  $Y|_{\mathcal M}$ is left invertible. Put $\mathcal H_0=\mathcal H\ominus\mathcal M$, $T_0=P_{\mathcal H_0} T|_{\mathcal H_0}$,  $\mathcal K_1=Y\mathcal M$,  $\mathcal K_0=\mathcal K \ominus\mathcal K_1$, $Y_2=P_{\mathcal K_1} Y|_{\mathcal H_0}$, $Y_0=P_{\mathcal K_0}Y|_{\mathcal H_0}$,  $V_1=V|_{\mathcal K_1}$, $V_2=P_{\mathcal K_1} V|_{\mathcal K_0}$, $V_0=P_{\mathcal K_0}V|_{\mathcal K_0}$,  
$$Z\colon\mathcal M\oplus\mathcal H_0\to\mathcal K_1\oplus\mathcal H_0, \ \  \ \ 
Z=\begin{pmatrix} Y|_{\mathcal M} & Y_2 \cr \mathbb O & I_{\mathcal H_0}\end{pmatrix},$$  and $$  R\colon\mathcal K_1\oplus\mathcal H_0\to\mathcal K_1\oplus\mathcal H_0, \ \  \ \ 
R= \begin{pmatrix} V_1 & V_2Y_0 \cr \mathbb O & T_0\end{pmatrix}.$$
Then $Y_0T_0=V_0Y_0$, $Z$ is invertible, $R=ZTZ^{-1}$, and $$(I_{\mathcal K_1}\oplus Y_0) R=V (I_{\mathcal K_1}\oplus Y_0) .$$ 
Moreover, if  $\operatorname{clos}Y \mathcal H=\mathcal K$, then $\operatorname{clos}Y_0 \mathcal H_0=\mathcal K_0$,  and if $\ker Y=\{0\}$, then $\ker Y_0=\{0\}$. 
\end{proposition}

\begin{lemma}\label{lemomega} Suppose  that $C>0$,  $g$, $\theta$, $\omega\in H^\infty$, $\theta$ is inner, 
and \begin{equation}\label{estp}\|P_-\overline\theta\omega u\|\leq C\|gu\| \ \ \text{ for every  } \ u\in H^2.\end{equation}
Furthermore, suppose   that $\mathcal H_0$ is a Hilbert space, and 
$X_0\colon\mathcal K_\theta\to\mathcal H_0$ is an operator. Put
\begin{equation}\label{defx}X\colon H^2\to\theta H^2\oplus\mathcal H_0, \ \  \ \  
X=\begin{pmatrix} g(S)|_{\theta H^2} & P_{\theta H^2}g(S)|_{\mathcal K_\theta} \cr\mathbb O & X_0\end{pmatrix}.\end{equation}
Then there exists $C_1>0$ (which depends on $C$, $\|g\|_\infty$, $\|\omega\|_\infty$, and $\|X_0\|$) such that 
$$ \|X\omega u\|\leq C_1\|gu\| \ \ \  \text{ for every } u\in H^2.$$
\end{lemma}
\begin{proof} Let $u\in H^2$. We have  
$$X\omega u=(gP_{\theta H^2}\omega u+P_{\theta H^2}gP_{\mathcal K_\theta}\omega u)\oplus X_0 
P_{\mathcal K_\theta}\omega u.$$
Therefore, $$\|X\omega u\|^2 = \|gP_{\theta H^2}\omega u+P_{\theta H^2}gP_{\mathcal K_\theta}\omega u\|^2+\| X_0P_{\mathcal K_\theta}\omega u\|^2.$$
We have $gP_{\theta H^2}\omega u+P_{\theta H^2}gP_{\mathcal K_\theta}\omega u =
g\omega u-P_{\mathcal K_\theta}gP_{\mathcal K_\theta}\omega u$,
$$
\|P_{\mathcal K_\theta}gP_{\mathcal K_\theta}\omega u\|\leq\|g\|_\infty\|P_{\mathcal K_\theta}\omega u\|,\ \  
\text{ and }\ \ \|P_{\mathcal K_\theta}\omega u\|=\|\theta P_-\overline\theta\omega u\|=\|P_-\overline\theta\omega u\|.
$$
Therefore, 
 \begin{align*}\|gP_{\theta H^2}\omega u & +P_{\theta H^2}gP_{\mathcal K_\theta}\omega u\|^2\leq
(\|g\omega u\|+\|P_{\mathcal K_\theta}gP_{\mathcal K_\theta}\omega u\|)^2\\&\leq
2(\|g \omega u\|^2+\|P_{\mathcal K_\theta}gP_{\mathcal K_\theta}\omega u\|^2)\leq
2(\|\omega\|_\infty^2\|g u\|^2+\|g\|_\infty^2\|P_-\overline\theta\omega u\|^2).\end{align*}
Thus,
\begin{align*}\|X\omega u\|^2 & \leq  2(\|\omega\|_\infty^2\|g u\|^2+\|g\|_\infty^2\|P_-\overline\theta\omega u\|^2)+
\|X_0\|^2\|P_-\overline\theta\omega u\|^2\\&=
2\|\omega\|_\infty^2\|g u\|^2+(2\|g\|_\infty^2+\|X_0\|^2)\|P_-\overline\theta\omega u\|^2.\end{align*}
It remains to apply  \eqref{estp}.
\end{proof}

In the following lemma, only part (i) is used in the sequel.

\begin{lemma}\label{lemomegaf} Suppose   that $g$, $\theta$, $\omega\in H^\infty$, and $\theta$ is inner. 
\begin{enumerate}[\upshape (i)]

\item If there exist $\eta$, $h\in H^\infty$ such that 
\begin{equation}\label{eq23} \omega=g\eta+\theta h, \end{equation}
then there exists $C>0$ such that \eqref{estp} is fulfilled.

\item If $g$ is outer and \eqref{estp} is fulfilled with some $C>0$, then  there exist $\eta$, $h\in H^\infty$ such that 
\eqref{eq23} is fulfilled.\end{enumerate}\end{lemma}
\begin{proof}(i) Let $u\in H^2$. We have $P_-\overline\theta\omega u=P_-\overline\theta(g\eta+\theta h)u=
P_-\overline\theta g\eta u$. Therefore, $\|P_-\overline\theta\omega u\|\leq \|g\eta u\|\leq\|\eta\|_\infty\|gu\|$.

(ii) 
For $\delta>0$ define an outer function $\psi_\delta$ by the formula
$$|\psi_\delta(\zeta)|=\begin{cases}1/|g(\zeta)|, & \text{ if }  |g(\zeta)|\geq \delta, \\
1, & \text{ if }  |g(\zeta)|<\delta.\end{cases} $$
Since $g$ is outer, 
\begin{equation}\label{lem231} (g\psi_\delta)(z)\to 1 \ \ \text{ for every } z\in\mathbb D, \ \ \text{ when } \delta\to 0.
\end{equation}
Clearly, $\|g\psi_\delta\|_\infty\leq\max(\|g\|_\infty, 1)$. 
Set $C_1=\max(\|g\|_\infty, 1)$. Applying \eqref{estp} with 
$\psi_\delta u$ instead of $u$, we obtain that $$\sup_{u\in H^2, \|u\|\leq 1}\|P_-\overline\theta\omega\psi_\delta u\|\leq CC_1 \ \ \text{ for every }\delta>0.$$
By the Nehari theorem (see, for example, {\cite[Theorem I.B.1.3.2]{nik}}), there exist $\varphi_\delta\in H^\infty$ such that 
$$\|\overline\theta\omega\psi_\delta-\varphi_\delta\|_\infty\leq CC_1 \ \ \text{ for every }\delta>0.$$ 
We have \begin{align*}\|g\varphi_\delta\|_\infty=\|\theta g\varphi_\delta\|_\infty\leq\|\theta g\varphi_\delta- \omega g \psi_\delta\|_\infty
 +   \|\omega g \psi_\delta\|_\infty & \leq \|g\|_\infty CC_1+\|\omega\|_\infty C_1\\& \text{ for every }\delta>0.\end{align*}
By the Vitali--Montel theorem, there exist a sequence $\delta_j\to_j 0$ and $h$, $\eta\in H^\infty$ such that 
\begin{equation}\label{lem232}(g\varphi_{\delta_j})(z)\to_j h(z)\ \text{ and }\ (\omega\psi_{\delta_j}-\theta\varphi_{\delta_j})(z) \to_j \eta(z) \ \ 
\text{ for every }z\in\mathbb D.\end{equation}
It follows from \eqref{lem231} and \eqref{lem232} that $\omega-\theta h=g\eta$.
\end{proof}

\begin{theorem}\label{theorem1} Suppose that  $g$, $\eta$, $h$, $\theta$, $\omega\in H^\infty$, 
$\theta$ and $\omega$ are inner, and $\omega = g\eta+\theta h$. Furthermore, suppose  that 
$\mathcal H_0$ is a Hilbert space, $T_0\colon\mathcal H_0\to\mathcal H_0$, 
$X_0\colon\mathcal K_\theta\to\mathcal H_0$, 
$Y_0\colon\mathcal H_0\to\mathcal K_\theta$ are operators, 
$Y_0X_0= g(T_\theta)$, 
$X_0 T_\theta=T_0X_0$, $Y_0T_0=  T_\theta Y_0$.
 Define $X$ by \eqref{defx},
put $Y=I_{\theta H^2}\oplus Y_0$, 
\begin{equation}\label{deft}T\colon\theta H^2\oplus\mathcal H_0\to\theta H^2\oplus\mathcal H_0, \ \ \ \  
T=\begin{pmatrix} S|_{\theta H^2} & P_{\theta H^2}S|_{\mathcal K_\theta}Y_0 \cr\mathbb O & T_0\end{pmatrix},\end{equation}
 and $\mathcal M=\operatorname{clos}X\omega H^2$. 
Then $\mathcal M\in\operatorname{Lat}T$, $T|_{\mathcal M}\approx  S$, and $Y\mathcal M=\vartheta\omega H^2$, 
where $\vartheta$ is the inner factor of $g$. Moreover,  
\begin{equation}\label{xyt} X  S=TX, \ \ \ YT=  S Y, \ \text{ and }
\ \ YX=  g(S).\end{equation} \end{theorem}
\begin{proof} The equalities \eqref{xyt} follow immediately from the definitions of $T$, $X$, and $Y$. The inclusion 
$\mathcal M\in\operatorname{Lat}T$ follows from the first equality in \eqref{xyt}. By Lemma~\ref{lemomegaf}, \eqref{estp} is fulfilled. 
We will show that there exists $c>0$ such that 
\begin{equation}\label{estxm} \|Yy\|\geq c\|y\| \ \ \text{ for every } y\in\mathcal M.\end{equation}
It is sufficient to prove  \eqref{estxm} for $y=X\omega u$, where $ u\in H^2$,
due to the definition of $\mathcal M$ and the continuity of $Y$. For such $y$, taking into account the third equality in \eqref{xyt} and the assumption that $\omega$ is inner,  the condition \eqref{estxm} can be rewritten as follows:
$$\|gu\| \geq c \|X\omega u\| \ \  \text{ for every } u\in H^2.$$
 This estimate is fulfilled by Lemma~\ref{lemomega}. Thus, \eqref{estxm} is proved. 

By \eqref{estxm}, the operator $Y|_{\mathcal M}\colon\mathcal M\to Y\mathcal M$ is invertible. By the second equality in 
\eqref{xyt}, $Y|_{\mathcal M}$ realizes the relation $T|_{\mathcal M}\approx  S|_{ Y\mathcal M}$. 
We have $$Y\mathcal M=Y\operatorname{clos}X\omega H^2=
\operatorname{clos}YX\omega H^2=\operatorname{clos}g(S)\omega H^2=
\vartheta\omega H^2,$$ 
where $\vartheta$ is the inner factor of $g$. Clearly, $ S|_{\vartheta\omega  H^2}\cong S$.
\end{proof}

The following simple lemma can be considered as a partial inverse to Theorem \ref{theorem1}.

 \begin{lemma}\label{leminvers} Suppose $g$, $\eta$, $h$, $\omega\in H^\infty$, $\omega$ is inner, and 
$1=g\eta+\omega h$. Furthermore, suppose $T\colon\mathcal H\to\mathcal H$ is a polynomially bounded operator, 
$Y\colon\mathcal H\to H^2$ is a quasiaffinity,  $(Y,S)$ is the isometric asymptote of $T$, 
$X\colon H^2\to\mathcal H$ is an operator such that $YX=g(S)$. 
Set $\mathcal M=\operatorname{clos}\omega(T)\mathcal H$. If $T|_{\mathcal M}\approx S$, then $T\approx S$.
\end{lemma}
\begin{proof} By assumption, $Y$ realizes the relation $T\prec S$. By \cite{mlak} or {\cite[Proposition 16]{ker16}}, $T$ is a.c.. 
Thus, the operator $\omega(T)$ is well defined. Since $Y\omega(T)=\omega(S)Y$ and $\operatorname{clos}Y\mathcal H = H^2$, we conclude that 
$\operatorname{clos}Y\mathcal M = \omega H^2$. Suppose that $T|_{\mathcal M}\approx S$. By {\cite[Theorem 1]{ker89}}, 
$Y|_{\mathcal M}$ is left invertible, in particular, $Y\mathcal M=\operatorname{clos}Y\mathcal M = \omega H^2$. 
Let $u\in H^2$. There exists $x\in\mathcal M$ such that $\omega u=Yx$. We have 
\begin{align*} u&=(g\eta+\omega h)u=g(S)(\eta u) + h(S)(\omega u)= (YX)(\eta u) + h(S)Yx\\&=Y(X(\eta u) + h(T)x).\end{align*}
We conclude that $Y \mathcal H = H^2$. Since $\ker Y=\{0\}$, $Y$ is invertible by the Closed Graph Theorem. \end{proof}
The following lemma is Nevanlinna's theorem, see {\cite[Corollary I.B.3.3.2]{nik}}.

\begin{lemma}[Nevanlinna's theorem]\label{aak} Suppose  that $\theta$, $g$, $\eta\in H^\infty$, $\|g\eta\|_\infty<1$, and $\theta$ is an inner function. 
Then there exists an inner function $\omega\in H^\infty$ such that $\omega-g\eta\in\theta H^\infty$.\end{lemma}

The following lemma is well known, see, for example, {\cite[Lemma 2]{sznagy}} or {\cite[Proposition X.1.1]{sfbk}}.

\begin{lemma}\label{lemrelprime} Suppose   that $\gamma$, $\theta$, $f\in H^\infty$,   $\gamma$ and $\theta$  are inner, 
 $\theta$ and the inner factor of $f$ are relatively prime.
 Then there exists $t\in\mathbb C$ such that $\gamma$ 
and the inner factor of $f +t\theta$ are  relatively  prime.\end{lemma}

\begin{lemma}\label{lem0} Suppose that $\gamma$, $\theta\in H^\infty$ are inner functions, $T_0$ is a polynomially bounded  operator,  $Y_0$ is  a quasiaffinity, and $Y_0T_0= T_\theta Y_0$. Then there exist  a quasiaffinity $X_0$ and 
$g\in H^\infty$ 
such that $X_0T_\theta=T_0X_0$, $Y_0X_0=g(T_\theta)$, and $\gamma$ and the inner factor of $g$ are relatively prime. \end{lemma}
\begin{proof}  By \cite{bercpr}, there exists a  contraction $R_0$ such that $R_0\prec T_0$. 
From the relations $R_0\prec T_0\prec T_\theta$ and the fact that $T_\theta$ is a $C_0$-contraction we conclude that 
$ T_\theta\prec R_0$  by {\cite[Theorem X.5.7]{sfbk}} or  {\cite[Proposition III.5.32]{berbook}}. We obtain that $T_0\sim T_\theta$. Denote by $X_0$ the quasiaffinity which realizes the relation $T_\theta\prec T_0$. We have $Y_0X_0\in \{T_\theta\}'$. 
By {\cite[Theorem X.2.10]{sfbk}} or  {\cite[Proposition III.1.21]{berbook}}, there exists 
$f\in H^\infty$ such that $Y_0X_0=f(T_\theta)$.  Since $f(T_\theta)$ is a quasiaffinity, $\theta$ and the inner factor of $f$ are relatively prime. By Lemma~\ref{lemrelprime}, there  exists $t\in\mathbb C$ such that $\gamma$ 
and the inner factor of $f +t\theta$ are relatively prime. Put $g=f +t\theta$. Then $Y_0X_0=g(T_\theta)$. 
 \end{proof}

\begin{lemma}\label{lemisoasymp} Suppose  that $T$ is a polynomially bounded operator and $T\prec S$. 
Then $S$ is the isometric asymptote of $T$.\end{lemma}

\begin{proof} Since $T$ is polynomially bounded, by \cite{bercpr}, there exists a contraction $R$ such that $R\prec T$. 
Clearly, $R\prec S$. By {\cite[Proposition 9]{kers}} or  {\cite[Lemma 2.1]{gam12}}, the isometric asymptote of $R$ is $S$. Denote by  $Y_R$ and $Y_T$ the canonical intertwining mappings for $R$ and $T$, and by $X$ and $Y$ the quasiaffinities which realize the relations $R\prec T$ and $T\prec S$, respectively. Let $V_T$ be the isometric asymptote of $T$.  
By {\cite[Theorem 1(a)]{ker89}}, there exist  operators $Z_1$  and $Z_2$ such that $Z_1S=V_TZ_1$,  $Y_T X=Z_1 Y_R$, 
$Z_2V_T=SZ_2$, and $Y=Z_2Y_T$.  Since the ranges of $Y_T$ and $Y$ are dense, the 
ranges of $Z_1$ and $Z_2$ are also dense. We obtain $S\buildrel d\over\prec V_T$ and $V_T\buildrel d\over\prec S$. 
By {\cite[Proposition 16]{ker16}}, $T$ is a.c.. Therefore, $V_T$ is a.c..  
Thus, $V_T\cong S$. \end{proof}

\begin{theorem}\label{theoremamain} Suppose that $T\colon\mathcal H\to\mathcal H$ is a polynomially bounded operator,  $T\prec S$, 
   $\mathcal M\in\operatorname{Lat}T$, and $T|_{\mathcal M}\approx S$. Then there exist  the operators 
$X_1$,$X_2\colon H^2\to\mathcal H $, $Y\colon\mathcal H\to H^2$, the functions $g_1$, $g_2\in H^\infty$ and   $\mathcal N\in\operatorname{Lat}T$ 
such that 
\begin{equation}\label{xyg} YT=SY, \  \   X_kS=TX_k, \ \ YX_k=g_k(S) \   \text{ and } \  \ X_kY=g_k(T), \ \ \ k=1,2,  \end{equation}
 the inner factors of $g_1$ and $g_2$ are relatively prime,  $T|_{\mathcal N}\approx S$, and 
$\mathcal M\vee\mathcal N=\mathcal H$. \end{theorem}
\begin{proof} By Lemma~\ref{lemisoasymp}, the isometric asymptote of $T$ is $S$. Denote by $Y$ the canonical 
 intertwining mapping. Then $Y|_{\mathcal M}$ is left invertible.  Let $\theta\in H^\infty$ be an inner function such that $Y\mathcal M=\theta H^2$. Let $T_0$ and $Y_0$ be defined as in Proposition~\ref{propa1}. Without loss of generality, we can suppose that $T$  has the form \eqref{deft} and $Y=I_{\theta H^2}\oplus Y_0$.   Then $Y_0$ is  a quasiaffinity which realizes the relation $T_0\prec T_\theta$.

 Denote by  $X_{10}$ and $g_1$  the quasiaffinity  and the function from  Lemma~\ref{lem0} applied to $\gamma=\theta$, 
$\theta$,  $T_0$, and $Y_0$. Denote by $\vartheta_1$ the inner factor of $g_1$. By Lemma~\ref{lem0}, $\theta$  and $\vartheta_1$ are relatively  prime. Denote by $X_1$ the operator defined by \eqref{defx} applied to $g_1$ and $X_{10}$. 

Take $\eta\in H^\infty$ such that $\|g_1\eta\|_\infty<1$, and $\theta$ and the inner factor of $\eta$ are relatively prime. By Lemma~\ref{aak}, there exist $\omega$, $h\in H^\infty$ such that $\omega$ is inner and $\omega=g_1\eta+\theta h$. Put $\mathcal N=\operatorname{clos}X_1\omega H^2$. By Theorem~\ref{theorem1}, $\mathcal N\in\operatorname{Lat}T$, $T|_{\mathcal N}\approx  S$, and 
$Y\mathcal N=\vartheta_1\omega H^2$. Furthermore,  \eqref{xyt} is fulfilled for $T$, $X_1$, $Y$, and $g_1$. Note that 
$\theta$ and $\vartheta_1\omega$ are  relatively prime. 

Denote by  $X_{20}$ and $g_2$  the quasiaffinity   and the function from  Lemma~\ref{lem0} applied to 
$\gamma=\vartheta_1$, $\theta$,  $T_0$, and $Y_0$. Denote by $\vartheta_2$ the inner factor of $g_2$. By Lemma~\ref{lem0}, $\vartheta_1$  and $\vartheta_2$ are relatively prime. Denote by $X_2$ the operator defined by \eqref{defx} applied to $g_2$ and $X_{20}$.  Then  \eqref{xyt} is fulfilled for $T$, $X_2$, $Y$, and $g_2$.  

It is easy to see from \eqref{deft} and \eqref{xyt} that 
$$ g(T)=\begin{pmatrix} g(S)|_{\theta H^2} & P_{\theta H^2} g(S)|_{\mathcal K_\theta}Y_0 \\ 
\mathbb O & g(T_0)\end{pmatrix}$$ for every $g\in H^\infty$.  The equations \eqref{xyg} follow from this equation 
 and \eqref{xyt}.  

It is well known that if \eqref{xyg} are fulfilled, and the inner factors of $g_1$ and $g_2$ are relatively prime,  then the mapping  $\mathcal J_Y$ defined by \eqref{lattice}
is a lattice isomorphism. We have 
$$ \mathcal J_Y(\mathcal M\vee\mathcal N)= \mathcal J_Y\mathcal M\vee\mathcal J_Y\mathcal N=\theta H^2\vee\vartheta_1\omega H^2=H^2$$
(because $\theta$ and $\vartheta_1\omega$ are  relatively prime). Therefore, $\mathcal M\vee\mathcal N=\mathcal H$.
  \end{proof}

\begin{remark}\label{remcontr} If $T$ is a contraction such that $T\prec S$, then there exists $\mathcal M\in\operatorname{Lat}T$ such that $T|_{\mathcal M}\approx S$ by \cite{ker07} or {\cite[Theorem IX.3.5]{sfbk}}. Thus, $T$ satisfies to the conditions of Theorem~\ref{theoremamain}.\end{remark}
 
\begin{lemma}\label{lemlat} Suppose that  $T$ is a polynomially bounded operator, 
$Y$ is a quasiaffinity, and $YT=SY$. 
Then the mapping $\mathcal J_Y$ defined by \eqref{lattice} 
is a lattice isomorphism.\end{lemma}
\begin{proof} By  \cite{bercpr}, there exists a contraction $R$ such that $R\prec T$. 
Clearly, $R\prec S$. By Remark~\ref{remcontr},  the conclusion 
of Theorem~\ref{theoremamain} is fulfilled for $R$. Using the relation $R\prec T$, it is easy to see that there exist operators $X_1$,$X_2\colon H^2\to\mathcal H $ and the functions $g_1$, $g_2\in H^\infty$  
such that the inner factors of $g_1$ and $g_2$ are relatively prime and  \eqref{xyg} 
  is fulfilled for $T$ and $Y$. Since the inner factors of functions $g_1$ and $g_2$ are relatively prime,  $\mathcal J_Y$ is a lattice-isomorphism.\end{proof}

\begin{remark} The relations \eqref{xyg} are proved for contractions $T$ such that  $T\prec S$ in 
 {\cite[Corollary 2.5]{gam03}} in other way. \end{remark}

The following lemma  follows from the definition of a.c. polynomially bounded operators (see {\cite[Lemma 2.2]{gam17}} for the proof).

\begin{lemma}\label{acperp} Suppose that $T$ is a polynomially bounded operator, and  
$\mathcal M\in\operatorname{Lat}T$. Then $T$ is a.c. if and only if $T|_{\mathcal M}$ and  
$ P_{\mathcal M^\perp}T|_{\mathcal M^\perp}$ are a.c..\end{lemma}

The following lemma can be easily checked directly, therefore, its proof is omitted.

\begin{lemma}\label{lemc1} Suppose that $T\colon\mathcal H\to\mathcal H$ is an operator, 
$\mathcal M$, $\mathcal N\in\operatorname{Lat}T$, and $\mathcal H=\mathcal M\vee\mathcal N$. 
Put $\mathcal E=\mathcal M\cap\mathcal N$. Then the quasiaffinity 
$P_{\mathcal M^\perp}|_{\mathcal N\ominus\mathcal E}$ realizes the relation 
$$P_{\mathcal N\ominus\mathcal E}T|_{\mathcal N\ominus\mathcal E}\prec P_{\mathcal M^\perp}T|_{\mathcal M^\perp}.$$\end{lemma}

The following proposition is closed to \cite{takdij} and {\cite[Proposition 20]{kers}}.

\begin{proposition}\label{propcap} Suppose that $T\colon\mathcal H\to\mathcal H$ is a polynomially bounded operator,  
$\mathcal M_1$, $\mathcal M_2\in\operatorname{Lat}T$, $T|_{\mathcal M_k}\prec S$, $k=1,2$, $\mathcal M_1\vee\mathcal M_2=\mathcal H$ 
and $\mathcal M_1\cap\mathcal M_2\neq\{0\}$. Then $T$ is a.c., $T\buildrel d\over\prec S$, and $S$ is the isometric asymptote of $T$. Moreover, if $T\not\prec S$, then $T|_{\ker Y}$ is a $C_0$-operator, where $Y$ is a canonical intertwining mapping. \end{proposition}
\begin{proof} Put $\mathcal E=\mathcal M_1\cap\mathcal M_2$. We can suppose that $\mathcal E\neq\mathcal M_2$. 
Denote by $X$ a quasiaffinity which realizes the relation $T|_{\mathcal M_2}\prec S$. There exists 
an inner function $\theta\in H^\infty$ such that $\mathcal J_X\mathcal E=\theta H^2$, where  $\mathcal J_X$ is defined by \eqref{lattice}. Put  
$$T_0=P_{\mathcal M_2\ominus\mathcal E}T|_{\mathcal M_2\ominus\mathcal E}, \ \  \ T_1= P_{\mathcal M_1^\perp}T|_{\mathcal M_1^\perp}, \ \  \text{ and } 
\ \ X_0=P_{\mathcal K_\theta}X_0|_{\mathcal M_2\ominus\mathcal E}.$$ Clearly, $X_0T_0=T_\theta X_0$.
Taking into account that $\mathcal J_X$ is a lattice isomorphism (by Lemma~\ref{lemlat}), it is easy to see that $X_0$ is a quasiaffinity. 
Therefore, $T_0$ is a.c. ({\cite[Proposition 16]{ker16}}) and $\theta(T_0)=\mathbb O$. 

By Lemma~\ref{lemc1}, $T_0\prec T_1$. By  \cite{mlak} or {\cite[Proposition 16]{ker16}} and Lemma~\ref{acperp}, $T$ is a.c.. Furthermore,  
\begin{equation}\label{thetaperp} \theta(T_1)=\mathbb O. \end{equation}

Let $(Y,U)$ be the unitary asymptote of $T$.  By {\cite[Theorem 3]{ker89}} and Lemma \ref{lemisoasymp},  
$U\cong U_{\mathbb T}$,  and $Y|_{\mathcal M_1}$ realizes the relation 
$T|_{\mathcal M_1}\prec U_{\mathbb T}|_{\xi H^2}$ for some $\xi\in L^\infty$ such that $|\xi|=1$ a.e. with respect to $m$. 

If $\operatorname{clos}Y\mathcal H=L^2$, then $P_{\xi H^2_-}Y|_{\mathcal M_1^\perp}$ realizes the relation 
$T_1\buildrel d\over\prec P_{\xi H^2_-}U_{\mathbb T}|_{\xi H^2_-}$, which contradicts to \eqref{thetaperp}. 
Thus, $\xi H^2\subset\operatorname{clos}Y\mathcal H\neq L^2$. Since $\operatorname{clos}Y\mathcal H\in\operatorname{Lat}U_{\mathbb T}$, 
we conclude that $\operatorname{clos}Y\mathcal H =\xi_1 H^2$ for some $\xi_1\in L^\infty$ such that $|\xi_1|=1$ a.e. with respect to $m$. Since $U_{\mathbb T}|_{\xi_1 H^2}\cong S$, we conclude that $S$ is the isometric asymptote of $T$. 

Put $\mathcal E_0=\ker Y$ and $\mathcal E_1=\operatorname{clos}P_{\mathcal M_1^\perp}\mathcal E_0$. 
Then $\mathcal E_0\in\operatorname{Lat}T$ and $\mathcal E_0\cap\mathcal M_1=\{0\}$. Therefore, 
 $\mathcal E_1\in\operatorname{Lat}T_1$ and $P_{\mathcal E_1}|_{\mathcal E_0}$ realizes the relation 
$T|_{\mathcal E_0}\prec T_1|_{\mathcal E_1}$. By \eqref{thetaperp}, $\theta(T|_{\mathcal E_0})=\mathbb O$. 
Thus, $T|_{\mathcal E_0}$ is a $C_0$-operator.
 \end{proof}


\section{Existence of a polynomially bounded operator with the minimal function $\exp(a\frac{z+1}{z-1})$ which is not similar to a contraction}

For $\lambda\in\Bbb D$ denote by $b_\lambda$  a Blaschke factor:
 $b_\lambda(z)=\frac{|\lambda|}{\lambda}
\frac{\lambda-z}{1-\overline\lambda z}$, $z\in\Bbb D$. For $w\in\Bbb D$ put 
\begin{equation}\label{beta}\beta_w(z) = \frac{w-z}{1-\overline w z}, \ \ \ z\in\Bbb D. \end{equation}
Clearly, $\beta_w\circ\beta_w(z) = z$ for every $z\in\Bbb D$.
 For every $\varphi\in H^\infty$ 
we have $\varphi\circ\beta_w\in H^\infty$,  
$\|\varphi\circ\beta_w\|_\infty = \|\varphi\|_\infty$, 
and 
$$b_\lambda\circ\beta_w = \zeta_{w,\lambda}b_{\beta_w(\lambda)}, \ \
\text{ where }   \zeta_{w,\lambda}\in\Bbb T.$$

For $a>0$ and  $0<\alpha<1$ put \begin{equation}\label{omegag}\omega_a(z)=\exp\Bigl(a\frac{z+1}{z-1}\Bigr) \  \text{ and }\  \ g_\alpha(z)=\exp\Bigl(-\Bigl(\frac{1+z}{1-z}\Bigr)^\alpha\Bigr), \ \ z\in\mathbb D.\end{equation}
We have $\omega_a$, $g_\alpha\in H^\infty$, $\omega_a$ is inner, and $g_\alpha$ is outer by {\cite[Example I.A.4.3.7, p. 71]{nik}}.

The following lemma can be easily proved by induction. Therefore, its proof is omitted.

\begin{lemma}\label{diffomegag} Let $a>0$, and  let $0<\alpha<1$. Let $\omega_a$ and $ g_\alpha$
be defined in \eqref{omegag}.
 Let $n\in\mathbb N$.
Then there exist functions $\psi_{\omega,n}$, $\psi_{g,n}$ analytic in $\mathbb D$ such that 
$$\omega_a^{(n)}=\omega_a\psi_{\omega,n} \ \ \text{ and }\ \ g_\alpha^{(n)}=g_\alpha\psi_{g,n}.$$
Furthermore, $$\psi_{\omega,n}(z)\!=\!\sum_{l=1}^{K_{\omega,n}}\frac{c_{\omega,n,l}}{(1-z)^{k_{\omega,n,l}}} 
 \text{ and } \psi_{g,n}(z)\!=\!\sum_{l=1}^{K_{g,n}}c_{g,n,l}\Bigl(\frac{1+z}{1-z}\Bigr)^{d_{1,l,n}}\!(1-z)^{d_{2,n,l}},$$
where $K_{\omega,n}$, $K_{g,n}$, $k_{\omega,n,l}\in\mathbb N$, $c_{\omega,n,l}$, $c_{g,n,l}$, $d_{1,l,n}$, 
$d_{2,n,l}\in\mathbb R$ .
 \end{lemma}

\begin{lemma}\label{diffmatrix}  Let $a>0$, and let  $0<\alpha<1$. Let $\omega_a$ and $g_\alpha$ be defined in
 \eqref{omegag}. Put
\begin{align*}\psi_{g,n,k} & = \frac{n!}{k!(n-k)!}\psi_{g,n-k}, \  \ \ 1\leq k\leq n, \\ 
\psi_{g,k,k} & =1, \ \psi_{g,n,k}=0, \  \ \ n\leq k-1, \ \  \psi_{\omega,0}=1. \end{align*}
For $M\in\mathbb N$  set $\Psi_M =[\psi_{g,n,k}]_{n,k=0}^{M-1}$ and define the functions 
$\kappa_{M,n}$, $0\leq n\leq M-1$, by the formula
$$ [\kappa_{M,n}]_{n=0}^{M-1}=\frac{\omega_a}{g_\alpha}\cdot\Psi_M^{-1}\cdot[\psi_{\omega,n}]_{n=0}^{M-1}.$$
Then $\kappa_{M,n}$ are  functions analytic in $\mathbb D$, and $\kappa_{M,n}(z)\to 0$ when $z\in(0,1)$, $z\to 1$ 
 for every $0\leq n\leq M-1$. \end{lemma}
\begin{proof} We have $\Psi_M$ is a lower triangular matrix-function, and the elements of the  main diagonal of $\Psi_M$ are equal to $1$. Therefore,  $\Psi_M$ is invertible. The elements of the matrix-function 
$\Psi_M^{-1}\cdot[\psi_{\omega,n}]_{n=0}^{M-1}$ have the form 
$$\sum_{l=1}^{K}c_l\Bigl(\frac{1+z}{1-z}\Bigr)^{d_{1,l}}(1-z)^{d_{2,l}}$$ for some $K\in\mathbb N$ and $c_l$, $d_{1,l}$, 
$d_{2,l}\in\mathbb R$ (depended on  $\alpha$, $a$, $M$, and the indices of the considered element of the matrix). 

Set $t=\frac{1+z}{1-z}$. Then $1-z=\frac{2}{1+t}$, and $\frac{\omega_a(z)}{g_\alpha(z)} =\exp(-at+t^\alpha)$. 
Furthermore, $t\to+\infty$ when $z\in(0,1)$, $z\to 1$. We have $t^{d_1}(1+t)^{d_2}\exp(-at+t^\alpha)\to 0$ when 
$t\to+\infty$ for every $d_1$, $d_2\in\mathbb R$. The conclusion of the lemma follows from this relation.
\end{proof} 

The following lemma is an emphasizing of {\cite[Lemma 3.5]{gam16}}.

\begin{lemma}\label{lemma3.5} 
Suppose  that  $\Lambda\subset\mathbb D$ 
is finite, $1\leq k_\lambda<\infty$ for every $\lambda\in\Lambda$,
and $B=\prod_{\lambda\in\Lambda}b_{\lambda}^{k_\lambda}$. 
Then there exists $C>0$ which depends on $B$ such that
$$\operatorname{dist}(\varphi, (B\circ\beta_w) H^\infty)\leq 
C\max_{\lambda\in\Lambda, 0\leq k\leq k_\lambda-1}|\varphi^{(k)}(\beta_w(\lambda))|$$
 for every $\varphi\in  H^\infty$ and  every $w\in\mathbb D$.
\end{lemma}
\begin{proof} We have $\operatorname{dist}(\varphi, (B\circ\beta_w) H^\infty)=
\operatorname{dist}(\varphi\circ\beta_w, B H^\infty)$  for every $ \varphi\in  H^\infty$ and  every $w\in\mathbb D$. 
By {\cite[Lemma 3.5]{gam16}}, there exists $C>0$ (which depends on $B$) such that
$$\operatorname{dist}(\psi, B H^\infty)\leq 
C\max_{\lambda\in\Lambda, 0\leq k\leq k_\lambda-1}|\psi^{(k)}(\lambda)|
\ \ \text{ for every } \psi\in  H^\infty.$$
We need to estimate $|(\varphi\circ\beta_w)^{(k)}(\lambda)|$ for $1\leq k\leq k_\lambda-1$.
By {\cite[Lemma 3.3]{gam16}},  there exist functions 
$c_{kl}\colon\mathbb D\to\mathbb C$
such that $\sup_{\mathbb D}|c_{kl}|<\infty$ for every $k\geq 1$, 
$0\leq l\leq k-1$,
and for every analytic function $\varphi\colon\mathbb D\to\mathbb C$ and every 
$w\in\mathbb D$
$$(\varphi\circ\beta_w)^{(k)}(z) = \sum_{l=1}^k
\varphi^{(l)}(\beta_w(z))\frac{c_{k,k-l}(w)}{(1-\overline wz)^{k+l}}.$$ Therefore, 
\begin{gather*}|(\varphi\circ\beta_w)^{(k)}(\lambda)|\leq
\sum_{l=1}^k|\varphi^{(l)}(\beta_w(\lambda))|\frac{|c_{k,k-l}(w)|}{|1-\overline w\lambda|^{k+l}}\\\leq\!
\max_{\lambda\in\Lambda, 0\leq k\leq k_\lambda-1}\!|\varphi^{(k)}(\beta_w(\lambda))| \cdot
\max_{\lambda\in\Lambda}(k_\lambda-1)\cdot \!
\sup_{{1\leq l\leq k, 1\leq k\leq k_\lambda-1,}\atop{w\in\mathbb D}}\!
\frac{|c_{k,k-l}(w)|}{(1-\max_{\lambda\in\Lambda}|\lambda|)^{k+l}}\end{gather*} for $1\leq k\leq k_\lambda-1$.
\end{proof}

\begin{theorem}\label{theoremnew} Suppose that  $C>0$, $a>0$,  $0<\alpha<1$, and $\{B_N\}_N$ is a sequence of finite Blaschke products with zeros from $(0,1)$. Then there exists a sequence $\{\rho_N\}_N\subset(0,1)$ such that 
for any $w\in[\rho_N,1)$ there exists $\eta_{N,w}\in H^\infty$ such that 
\begin{equation}\label{incb}\operatorname{dist}(\eta_{N,w}, (B_N\circ\beta_w) H^\infty)\leq C \ \ \text{  and } \ \ \omega_a-g_\alpha\eta_{N,w}\in (B_N\circ\beta_w) H^\infty,\end{equation}
where $\beta_w$ is defined in \eqref{beta}, and $\omega_a$ and $g_\alpha$ are defined in \eqref{omegag}. 
\end{theorem}
\begin{proof} Let $N$ be fixed. Denote by $C_N$ a constant  from Lemma~\ref{lemma3.5} applied to 
$B_N=\prod_{\lambda\in\Lambda}b_{\lambda}^{k_\lambda}$. 
By Lemma~\ref{diffmatrix}, there exists $\rho_{1N}\in(0,1)$ such that
$|\kappa_{k_\lambda,k}(z)|\leq C/C_N$ for every
$z\in[\rho_{1N},1)$, $0\leq k\leq k_\lambda-1$, $\lambda\in\Lambda$. 

Since $\beta_w(\lambda)\to 1$ when $w\to 1$ for every $\lambda \in \mathbb D$, 
 and $\beta_w(\lambda)\in(0,1)$ when $0<\lambda<w<1$, there exists 
$\rho_N\in(0,1)$ such that   $\beta_w(\lambda)\in[\rho_{1N},1)$ when $w\in[\rho_N,1)$ for every $\lambda\in\Lambda$. 

 Since $\Lambda$ is a finite set, and $k_\lambda<\infty$ for every $\lambda\in\Lambda$, we have that for every $w\in\mathbb D$  there exists $\eta_{N,w}\in H^\infty$ such that 
\begin{equation}\label{eta}\eta_{N,w}^{(k)}(\beta_w(\lambda))=\kappa_{k_\lambda,k}(\beta_w(\lambda))\ \text{ for every } \ \lambda\in\Lambda \ \text{ and }\ 0\leq k\leq k_\lambda-1.\end{equation} 
By Lemma~\ref{lemma3.5}, $\operatorname{dist}(\eta_{N,w}, (B\circ\beta_w) H^\infty)\leq C$ for $w\in[\rho_N, 1)$. 

The inclusion in \eqref{incb} is fulfilled if and only if  
\begin{equation}\label{diffb}\omega_a^{(k)}(\beta_w(\lambda))-(g_\alpha\eta_{N,w})^{(k)}(\beta_w(\lambda))=0 \ \ \text{ for all } 
\lambda\in\Lambda,  \ 0\leq k\leq k_\lambda-1.\end{equation}
Clearly, $$(g_\alpha\eta_{N,w})^{(n)}(z)=\sum_{k=0}^n\frac{n!}{k!(n-k)!}g_\alpha^{(n-k)}(z)\eta_{N,w}^{(k)}(z) \ \text{ for } \ n\in\mathbb N \ \text{ and } z\in\mathbb D.$$
Therefore, $$[(g_\alpha\eta_{N,w})^{(k)}(z)]_{k=0}^{k_\lambda-1}=g_\alpha(z)\cdot\Psi_{k_\lambda}(z)\cdot[\eta_{N,w}^{(k)}(z)]_{k=0}^{k_\lambda-1}$$ for every $z\in\mathbb D$, where $\Psi_{k_\lambda}$ is defined in Lemma~\ref{diffmatrix}. 
By \eqref{eta} and  Lemmas~\ref{diffomegag} and \ref{diffmatrix}, 
\begin{align*}[(g_\alpha\eta_{N,w})^{(k)}(z)]_{k=0}^{k_\lambda-1} &=g_\alpha(z)\cdot\Psi_{k_\lambda}(z)\cdot[\kappa_{k_\lambda,k}(z)]_{k=0}^{k_\lambda-1}= \omega_a(z)\cdot[\psi_{\omega,k}(z)]_{k=0}^{k_\lambda-1} \\ &  = [\omega_a^{(k)}(z)]_{k=0}^{k_\lambda-1}\ \ \text{ when } \ z=\beta_w(\lambda), \  \lambda\in\Lambda. \end{align*} 
Thus, the equalities \eqref{diffb} are fulfilled, and  the inclusion in \eqref{incb} is proved.
\end{proof}

The following theorem is  a version of {\cite[Theorem 3.8]{gam16}}. The condition \eqref{eq3.6} from the theorem is 
the {\it generalized Carleson condition}, see {\cite[Theorem II.C.3.2.14, p. 164]{nik}}.

\begin{theorem}\label{theorem3.8} Suppose  that $C>0$, $\{\rho_N\}_N\subset(0,1)$, 
$g \in H^\infty$, and $g^{(k)}(r)\to 0$ when 
$r\in(0,1)$, $r\to 1$, for every $k\geq 0$.
Furthermore, suppose  that 
 $B_N$ are finite Blaschke products with zeros from $(0,1)$, 
$\varphi_N \in H^\infty$, $\varphi_N(\lambda)\neq 0$ 
for every $\lambda\in\Bbb D$ such that $B_N(\lambda)=0$, 
 for every index $N$.
Then there exist $\delta>0$ and sequences of $w_N\in[\rho_N,1)$, 
of $\zeta_N\in\mathbb T$, and 
of $\psi_N \in H^\infty$ such that 
the product $\prod_N\zeta_NB_N\circ\beta_{w_N}$ converges,
\begin{equation}\label{eq3.6}\bigl|\prod_N\zeta_N(B_N\circ\beta_{w_N})(z)\bigr|\geq\delta\inf_N|(B_N\circ\beta_{w_N})(z)|
 \ \ \text{ for every } \ z\in\mathbb D, \end{equation}
$g-\psi_N\cdot\varphi_N\circ\beta_{w_N}\in (B_N\circ\beta_{w_N}) H^\infty$,
 and 
$\operatorname{dist}(\psi_N, (B_N\circ\beta_{w_N}) H^\infty)\leq C$.
\end{theorem}
\begin{proof} The proof of {\cite[Theorem 3.8]{gam16}} is beginned from the definition of the sequence
$\{r_N\}_N\subset(0,1)$. To prove the present theorem, it needs to replace $r_N$ by $\max(r_N,\rho_N)$ and to repeat 
the proof of  {\cite[Theorem 3.8]{gam16}}.
\end{proof}




The following theorem is a version of {\cite[Theorem 7.1]{gam16}}.

\begin{theorem}\label{theorem7.1} Let $a>0$, and let  $0<\alpha<1$. Let $\omega_a$ and $g_\alpha$ be defined in
\eqref{omegag}. Then there exist a function $\eta\in H^\infty$, 
an operator $T\colon\mathcal H\to\mathcal H$,  
a Blaschke product $B$ with zeros from $(0,1)$, 
and quasiaffinities $X\colon\mathcal H\to\mathcal K_B$, $Y\colon\mathcal K_B\to\mathcal H$
such that \begin{equation}\label{eqnew}\omega_a-g_\alpha\eta\in BH^\infty,\end{equation}
$T$ is polynomially bounded, $T$ is not similar to a contraction, 
\begin{equation}\label{eqtheorem7.1}XT=T_BX,  \ \ \ YT_B=TY, \ \ \text{ and } \ \ XY=g_\alpha(T_B),\end{equation} 
where $T_B$ is defined in \eqref{1.7}.\end{theorem}
\begin{proof}
Let $\{T_N\}_N$ and $\{B_N\}_N$ be
the sequences of operators 
 and of finite Blaschke products with zeros
from $(0,1)$ from {\cite[Theorem 6.3]{gam16}}, respectively. 
Let $C>0$ be fixed. Denote by $\mathcal H_N$ 
the finite dimensional spaces on which $T_N$ acts. 
There exist invertible
operators $X_N\colon\mathcal H_N\to\mathcal  K_{B_N}$, 
$Y_N\colon\mathcal  K_{B_N}\to\mathcal  H_N$ such that
$X_NT_N=T_{B_N}X_N$, $Y_NT_{B_N}=T_NY_N$,
$\|X_N\|\leq C$, $\|Y_N\|\leq C$. 
By [SFBK, Theorem X.2.10], there exist functions $\varphi_N\in H^\infty$
such that $X_NY_N=\varphi_N(T_{B_N})$. 
Note that 
 $\varphi_N(\lambda)\neq 0$ for every $\lambda\in\Bbb D$
such that $B_N(\lambda)= 0$, for every index $N$.

Let $\{\rho_N\}_N$ be the sequence from Theorem~\ref{theoremnew} applied to $C$, $a$, $\alpha$, and $\{B_N\}_N$.
 Applying Theorem~\ref{theorem3.8} to $C$, $\{\rho_N\}_N$, $g_\alpha$, sequences of  $B_N$ and of $\varphi_N$ we obtain 
$\delta>0$ and sequences of $w_N\in[\rho_N,1)$, of
$\zeta_N\in\Bbb T$ and of $\psi_N \in H^\infty$
 which satisfy the conclusion of  Theorem~\ref{theorem3.8}.
Put $$B=\prod_N\zeta_NB_N\circ\beta_{w_N} \ \ \text{ and } 
\ \ T=\bigoplus_N\beta_{w_N}(T_N).$$ 
By {\cite[Theorem 6.3]{gam16}} and {\cite[Corollary 1.2]{gam16}}, $T$ is polynomially bounded, 
and $T$ is not similar to a contraction. 

Let  $\eta_{N,w_N}$ be  from  Theorem~\ref{theoremnew}. Then \eqref{incb} is fulfilled for every $N$  with $w=w_N$. 
By \eqref{eq3.6},  \eqref{incb}, and {\cite[Theorem II.C.3.2.14, p. 164]{nik}}, there exists a function  $\eta\in H^\infty$ 
such that $\eta-\eta_{N,w_N}\in (B_N\circ\beta_{w_N})H^\infty$ for every $N$. We have 
$$\omega_a-g_\alpha\eta=\omega_a-g_\alpha\eta_{N,w_N}+g_\alpha(\eta_{N,w_N}-\eta)\in (B_N\circ\beta_{w_N}) H^\infty$$ for every $N$ due to \eqref{incb}. Thus,  \eqref{eqnew} is proved.

Equalities \eqref{eqtheorem7.1} are proved exactly as in the proof of {\cite[Theorem 7.1]{gam16}}.
\end{proof}

\begin{theorem}\label{example} Let $a>0$, and let   $\omega_a$ be defined in \eqref{omegag}. 
 Then there exists a cyclic a.c.  polynomially bounded operator $T_0$ such that $T_0$ is not similar to a contraction and $\omega_a(T_0)=\mathbb O$.\end{theorem}

\begin{proof} Take $0<\alpha<1$, and apply Theorem~\ref{theorem7.1} to $a$ and $\alpha$. Denote the 
polynomially bounded operator and quasiaffinities constructed in  Theorem~\ref{theorem7.1} by $T_{10}$, $X_{10}$, and $Y_{10}$, respectively. Define $X$ and $T$ by \eqref{defx} and \eqref{deft},  respectively, with $\theta=B$, where $B$ is the Blaschke product obtained in  Theorem~\ref{theorem7.1}. By {\cite[Corollary 2.3]{gam16}}, $T$ is polynomially bounded, and $T$ is not similar to a contraction. Set $Y=I_{BH^2}\oplus Y_{10}$. Clearly, $Y$ is a  quasiaffinity. Since $g_\alpha$ is outer, $X$  is a  quasiaffinity (sf. {\cite[Lemma 2.1 and Proposition 2.4]{gam16}}). Taking into account  \eqref{xyt}, we obtain that  $T\sim S$. Therefore, $T$ is cyclic, and, by \cite{mlak} or {\cite[Proposition 16]{ker16}}, $T$ is a.c.. 

Put $$\mathcal M=\operatorname{clos}X\omega_a H^2.$$ By Theorem~\ref{theorem1}, $T|_{\mathcal M}\approx S$, and 
$Y\mathcal M=\omega_a H^2$ (because $g_\alpha$ is outer). 
Put $$T_0=P_{\mathcal M^\perp}T| _{\mathcal M^\perp}.$$
Then $T_0$ is cyclic, because $T$ is cyclic (see \eqref{muorth}), and $T_0$ is a.c., because $T$ is a.c. (\cite{mlak} or {\cite[Propositions 14 and 35]{ker16}}). It is easy to see that $\omega_a(T_0)=\mathbb O$. Indeed, denote by 
$\mathcal H$ the space in which $T$ acts. We have
$$\mathcal M=\operatorname{clos}X\omega_a H^2=\operatorname{clos}X\omega_a(S) H^2=
\operatorname{clos}\omega_a(T)X H^2=\operatorname{clos}\omega_a(T)\mathcal H.$$
Therefore,  $\omega_a(T_0)=P_{\mathcal M^\perp}\omega_a(T)| _{\mathcal M^\perp}=\mathbb O$. 

Finally, $T_0$ is not similar to a contraction by {\cite[Corollary 4.2]{cass}}. Indeed, if $T_0$ is similar  to a contraction, we can apply {\cite[Corollary 4.2]{cass}} to $T$, because $T|_{\mathcal M}$ is similar to an isometry, and $T$ is power bounded. Then $T$ must be  similar  to a contraction, a contradiction.\end{proof}

.

\end{document}